\documentclass[a4paper,12pt]{amsart}
\usepackage{amssymb,amsfonts,amsmath}
\usepackage{ifthen}
\usepackage{graphicx}
\usepackage{float}
\usepackage{cite}
\newtheorem{theorem}{Theorem}[section]
\newtheorem{lemma}[theorem]{Lemma}
\newtheorem{corollary}[theorem]{Corollary}

\newtheorem{definition}[theorem]{Definition}
\newtheorem{remark}[theorem]{Remark}

\numberwithin{equation}{section}

\newcounter{minutes}
\divide\time by 60
\newcounter{hours}
\multiply\time by 60
\addtocounter{minutes}{-\time}
\newcommand{\D}{{\mathbb D}}

\newcommand{\real}{{\operatorname{Re}\,}}

\newcommand{\ds}{\displaystyle}

\begin{document}
%\hfill{File}

\bibliographystyle{amsplain}

\title[On a generalization of starlike functions]%
{On a generalization of starlike functions}

%%%%%%%% BEGIN TIMESTAMP
\def\thefootnote{}
\footnotetext{ \texttt{\tiny File:~\jobname .tex,
          printed: \number\day-\number\month-\number\year,
          \thehours.\ifnum\theminutes<10{0}\fi\theminutes}
} \makeatletter\def\thefootnote{\@arabic\c@footnote}\makeatother
%%%%%%%% END TIMESTAMP

\author{Sarita Agrawal}%${}^{\#}$}

\thanks{%${}^\#$ The corresponding author\\
Postdoctoral Fellow,
Institute of Mathematical Sciences (IMSc) Chennai, IV Cross Road, CIT Campus, Taramani, Chennai 600 113, India\\
{\em Email: saritamath44@gmail.com}\\
}
\begin{abstract}
In this paper, we  mainly study the order of $q$-starlikeness of the well-known basic hypergeometric function. In addition, we obtain the Bieberbach-type problem for a generalized class of starlike functions. We also discuss the Fekete-szeg\"o and the Hankel determinant problems for the same class of functions.


  
\end{abstract}

\maketitle
\pagestyle{myheadings}
\markboth{Sarita Agrawal}{On a generalization of starlike functions}

\section{Introduction}\label{Intro+mainthms}
Without loss of generality, the study of univalent functions has been restricted to the open unit disk $\mathbb{D}=\{z\in\mathbb{C}:\,|z|<1\}$. 
The subclasses of the class of univalent functions such as the classes of convex, starlike, and close-to-convex 
functions have been extensively investigated in the literature by numerous researchers in this field of study. We refer \cite{Dur83,Goo83} for the basic results on these classes.

We denote by $\mathcal{H}(\D)$, the set of all functions analytic (or holomorphic) in $\D$. We use the symbol $\mathcal{A}$ for the class of functions 
$f \in \mathcal{H}(\D)$ with 
the Montel normalization $f(0)=0=f'(0)-1$. This means that the functions 
$f\in\mathcal{A}$ have the power series representation of the form 
\begin{equation}\label{e1}
f(z)=z+\sum_{n=2}^\infty a_nz^n
\end{equation}
where $a_n$'s are complex numbers. Let $\mathcal{S}$ denote the class of all univalent functions in $\mathcal{A}$. 

A function $f\in \mathcal{A}$ is said to be starlike of order $\alpha$, $0\le \alpha<1$, if 
\begin{equation}\label{starlike-eq}
{\rm Re}\,\left(\frac{zf'(z)}{f(z)}\right)>\alpha,\quad z\in \mathbb{D}.
\end{equation}
We use the notation $\mathcal{S}^*(\alpha)$ for the class of all starlike functions of order
$\alpha$. This is immediate that for $\alpha=0$, this class coincides with the class $\mathcal{S}^*$ of all starlike functions. 
 
In \cite{AS17}, the authors generalize the class $\mathcal{S}^*(\alpha)$ by replacing the derivative in \eqref{starlike-eq} by the $q$-derivative $D_q(f)$ (also called the $q$-difference operator), defined for $0<q<1$ by
the equation
\begin{equation}\label{sec1-eqn1}
(D_qf)(z)=\frac{f(z)-f(qz)}{z(1-q)},\quad z\neq 0, \quad (D_qf)(0)=f'(0)
\end{equation}
and the right half plane by a suitable domain. They defined that a function
$f\in \mathcal{S}^*_q(\alpha)$, $0\le \alpha<1$, if
\begin{equation}\label{def1}
\left |\frac{\ds\frac{zD_qf(z)}{f(z)}-\alpha}{1-\alpha}-\frac{1}{1-q}\right|\le \frac{1}{1-q}, \quad z\in \D.
\end{equation}
Set $\mathcal{S}_q^*:=\mathcal{S}_q^*(0)$. The class $\mathcal{S}_q^*$ is defined and studied by Ismail in \cite{IMS90}.  As $q\to1^-$, $\mathcal{S}_q^*(\alpha)$ is nothing but $\mathcal{S}^*(\alpha)$, $0\le \alpha<1$. Definition~\ref{def1} implies that ${\rm Re}\,((z(D_qf)(z)/f(z))-\alpha)/(1-\alpha)$ lies in a disk centred at $1/(1-q)$ with the radius $1/(1-q)$. This implies
$${\rm Re}\,\left(\ds\frac{\ds\frac{z(D_qf)(z)}{f(z)}-\alpha}{1-\alpha}\right)>0
$$
or,
\begin{equation}\label{e1.5}
{\rm Re}\,\left(\frac{z(D_qf)(z)}{f(z)}\right)>\alpha,\quad z\in \mathbb{D}.
\end{equation}
But (\ref{e1.5}) need not imply (\ref{def1}).
Hence, it is natural to consider the following class in the more general way.
\begin{definition}
A function $f\in \mathcal{A}$ is said to be in the class of $q$-starlike functions of order $\alpha$, denoted by $\mathcal{S}_q(\alpha),\, 0\le \alpha<1$, if
$${\rm Re}\,\left(\frac{z(D_qf)(z)}{f(z)}\right)>\alpha,\quad z\in \mathbb{D}.
$$
\end{definition}
Clearly, $\mathcal{S}_q(\alpha)\supset\mathcal{S}_q^*(\alpha)$ and as $q\to 1^-$, $\mathcal{S}_q(\alpha)=\mathcal{S}^*(\alpha)$. Set $\mathcal{S}_q:=\mathcal{S}_q(0)$ so that as $q\to 1^-$, $\mathcal{S}_q(0)=\mathcal{S}^*$.

In Section~\ref{sec2}, we establish a result on the order of $q$-starlikeness of shifted basic hypergeometric functions
$$z\Phi[a,b;c;q,z]=z\sum_{n=0}^\infty \frac{(a;q)_n(b;q)_n}{(c;q)_n(q;q)_n}z^n, \quad
z\in\D,
% = 1+\frac{(1-a)(1-b)}{(1-c)(1-q)}z+\frac{(1-a)(1-aq)(1-b)(1-bq)}{(1-c)(1-cq)(1-q)(1-q^2)}z^2+\cdots,
$$
where $(a;q)_n=(1-a)(1-aq)(1-aq^2)\cdots (1-aq^{n-1})$, $(a;q)_0=1$ with $0\le q<1$, $a,b,c$ are real parameters and $(c;q)_n\neq 0$. For the basic properties of Heine's hypergeometric functions the basic references are \cite{AAR99,Sla66}. 
Interestingly, the replacements of $a,b$ and $c$ by $q^a,q^b$ and $q^c$ respectively, then as ${q\to 1^-}$, the function $\Phi[q^a,q^b;q^c;q,z]$ tends to the well-known Gaussian hypergeometric functions
$$zF(a,b;c;z)=z\sum_{n=0}^\infty\frac{(a)_n(b)_n}{(c)_n(1)_n}z^n,\quad
z\in\D,
$$
where $a, b, c$ are real parameters, $(a)_0=1$, $(a)_n=a(a+1)\cdots(a+n-1)$ is the Pochhammer symbol and $c$ is neither $0$ nor a negative integer (except in special cases where $a=-m$ or $b=-m$ and $c=-p$ with $p>m$). 

\section{Order of $q$-starlikeness of basic hypergeometric functions}\label{sec2}
In \cite{Kus02}, K\"{u}stner investigated the order of starlikeness of the shifted Gaussian hypergeometric functions. In a generalized way he extended the notion of order of starlikeness of $f\in \mathcal{A}$ as
$$\sigma(f):=\inf_{z\in \mathbb{D}}{\rm Re}\,\left(\frac{zf'(z)}{f(z)}\right)\in [-\infty,1]
$$ 
and established that the order of starlikeness of shifted Gaussian hypergeometric functions $zF(a,b;c;z)$ is $-\infty$ under certain constraints on the real parameters $a, b, c$. Basic results on the order of starlikeness of the Gaussian hypergeometric functions can be found in \cite{HPV10,Kus02,MM90,Pon97,PV01,Sil93}.  

Now, let us define the order of $q$-starlikeness of $f\in\mathcal{A}$ as, 
$$\sigma_q(f)=\inf_{z\in \D} \real \left(\frac{z(D_qf)(z)}{f(z)}\right) \in [-\infty,1].
$$
Clearly, $\sigma_q(f)=1$ for $f(z)=z$. Note that $\lim_{q\to 1}\sigma_q(f)=\sigma(f)$.
  
Here our main aim is to consider the shifted basic 
hypergeometric functions and determine the order of $q$-starlikeness in the sense of K\"ustner under some conditions on $a, b, c$ and $q$ as follows:

\begin{theorem}\label{sec3-thm1}
Let $a,b,c$ be non-negative real numbers with $0<1-aq< 1-cq$ and $0<1-b< 1-c$. 
For $0<q<1$  and $r \in (0,1]$, 
the function $z\mapsto z\Phi[a,b;c;q,rz]$ has the order of $q$-starlikeness
$$\sigma_q(z\Phi[a,b;c;q,rz])=1+\rho q\frac{(1-a)(1-b)}{(1-c)(1-q)}\,\frac{\Phi[aq,bq;cq;q,\rho]}
{\Phi[a,b;c;q,\rho]}
$$
where 
$$ \rho=-r \mbox{ if } \frac{q(1-a)}{a(1-q)}=:s > 0 \mbox{ and } \rho=r \mbox{ if } s < 0.
$$
In particular, we have
$$1+\frac{s\rho}{1-\rho}\le \sigma_q(z\Phi[a,b;c;q,rz]) \le 1+\frac{\rho s(1-b)}{2(1-c)}.
$$
\end{theorem}
From this we can observe the following
\begin{remark}
The case $s<0$ with $r=1$ in Theorem~\ref{sec3-thm1} is considered in the limiting sense.
In this case, the lower bound $1+\displaystyle\frac{sr}{1-r}$ is equal to $-\infty$.
\end{remark}

\begin{proof}[\bf Proof of Theorem~\ref{sec3-thm1}]
Set $\Phi(z)=\Phi[a,b;c;q,z]$ and $f(z)=z \Phi(z)$. 
Now, by (\ref{sec1-eqn1}) we have
\begin{eqnarray*}
(D_qf)(z)&=&\frac{\Phi(z)-q\Phi(qz)}{1-q}\\
&=&\frac{\Phi(z)-q\Phi(z)+q\Phi(z)-q\Phi(qz)}{1-q}\\
&=&\frac{\Phi(z)(1-q)+q(\Phi(z)-\Phi(qz))}{1-q}\\
&=&\Phi(z)+zq(D_q \Phi)(z).
\end{eqnarray*}
Hence,
\begin{equation}\label{eqn0.2}
w=\frac{z(D_qf)(z)}{f(z)}=1+zq\frac{(1-a)(1-b)}{(1-c)(1-q)}\,\frac{\Phi[aq,bq;cq;q,z]}
{\Phi[a,b;c;q,z]},
\end{equation}
where the last equality holds by \cite[1.12(ii), pp. 27]{GR90}. Recall the difference equation stated in \cite{AS14}, which is equivalent to
$$\frac{\Phi[aq,bq;cq;q,z]}{\Phi[a,b;c;q,z]}=\frac{(1-c)}{a(1-b)z} \left[\frac{\Phi[aq,b;c;q,z]}{\Phi[a,b;c;q,z]}-1\right].
$$

Substituting this ratio in (\ref{eqn0.2}), we get
$$w = 1+s\left[\frac{\Phi[aq,b;c;q,z]}{\Phi[a,b;c;q,z]}-1\right]
= 1-s+s\frac{\Phi[aq,b;c;q,z]}{\Phi[a,b;c;q,z]},
$$
where $s$ is defined in the statement of our theorem with $q \in (0,1)$.
It follows from \cite{AS14} that $w$ has an integral representation
\begin{equation}\label{eqn0.3}
w=1-s+s\int_0^1 \frac{1}{1-tz}\mbox{d}\mu(t),
\end{equation}
with the non-negative real numbers $a,b,c$ satisfying the conditions 
$0\le 1-aq \le 1-cq$ and $0<1-b<1-c$.
Now, for $ s>0$, $r\in (0,1]$ and from equation (\ref{eqn0.3}) it follows that the minimum of 
$\real w$ for $|z|\le r$ is attained at the point $z=-r$ and that the minimum is 
$1-\ds \frac{rs}{(1+r)}$.
Secondly, for $s<0$, $r\in (0,1]$ and from equation (\ref{eqn0.3}), 
it follows that the minimum of $\real w$ for $|z|\le r$ is attained at 
the point $z=r$ and that the minimum is $1+\ds \frac{rs}{(1-r)}$.
This in combination with (\ref{eqn0.2}), yields the order of $q$-starlikeness of
$z\Phi[a,b;c;q,rz]$.

The upper estimate for ${\rm Re}\,w$
follows from (\ref{eqn0.2}) and an integral representation of the ratio 
${\Phi[aq,bq;cq;q,z]}/{\Phi[a,b;c;q,z]}$ obtained in \cite[Theorem~2.13]{AS14}. 
Hence, the conclusion of our theorem follows immediately.
\end{proof}

\begin{remark}
Making the substitutions $a\to q^a$, $b\to q^b$ and $c\to q^c$, and taking the limit
as $q\to 1^-$, we achieve the result of K\"ustner \cite[Theorem~1.1]{Kus02} as special case. 
\end{remark}

\begin{corollary}\label{cor1}
Let $q\in (0,1)$, $a, b, c$ be non-negative real numbers such that $1>a\ge b\ge c$ and $r\in (0,1]$. Then
$$z\varPhi[a,b;c;q,rz]\in \mathcal{S}_q\left(1-\ds \frac{rs}{1+r}\right),
$$
where $s=\ds\frac{q(1-a)}{a(1-q)}$.
\end{corollary}
\begin{proof}
The condition $0<a<1$ gives $s>0$. The case $s>0$ in Theorem~\ref{sec3-thm1} concludes the result, since the conditions on the parameters $a, b$, and $c$ in Theorem~\ref{sec3-thm1} and Corollary~\ref{cor1} are same, when $0<a<1$.
\end{proof}

On substituting of $r=1$ in Corollary~\ref{cor1} gives the order of $q$-starlikeness of the shifted basic hypergeometric functions with certain conditions on $a, b$, and $c$.

\begin{corollary}\label{cor2}
Let $a, b, c$ be non-negative real numbers such that $1>a\ge b\ge c$. Then
$$z\varPhi[a,b;c;q,z]\in \mathcal{S}_q(1-s/2),
$$
where $s=\ds\frac{q(1-a)}{a(1-q)}$.
\end{corollary}

In particular, when $s=2$, the shifted basic hypergeometric function $z\varPhi[a,b;c;q,z]\in \mathcal{S}_q$. 

When we put $a=q^a,b=q^b$ and $c=q^c$ and allow $q\to 1^-$, then Corollary~\ref{cor2} leads to an well-known result on the order of starlikeness of the Gaussian hypergeometric functions namely,

\begin{corollary}\cite[Theorem~B]{RS86}\label{cl}
Let $a, b, c$ be non-negative real numbers such that $a\le b\le c$. Then
$$zF(a,b;c;z)\in \mathcal{S}^*(1-a/2).
$$
\end{corollary}

This result on Gaussian hypergeometric function of Corollary~\ref{cl} is not only interesting by itself, but also useful and employed for further research in geometric function theory. Many researchers used this result, particularly Ponnusamy and Sahoo in \cite{PS08} used it innovatively to study pre-Schwarzian norm estimates for integral operators (defined by convolution) of functions belonging to special subclasses of the class of univalent functions with hypergeometric functions. Similarly, we expect that these results on basic hypergeometric functions will be fruitful and pave the way for further research in function theory as well as in physics. 

\section{Coefficient estimates}
In geometric function theory, finding bound for the coefficient $a_n$ of functions of the form (\ref{e1}) is an important
problem, as it reveals the geometric properties of the corresponding function. For
example, the bound for the second coefficient $a_2$ of functions in the class $\mathcal{S}$, gives the growth and distortion properties as well as covering theorems. Estimating coefficient of functions from the class of univalent functions and its subclasses is of interest among function theorists since last decade; see for instance \cite{AS17,deB85}. 

Another interesting coefficient estimation is the Hankel determinant. The $k^{th}$ order Hankel determinant ($k\ge 1$) of $f\in 
\mathcal{A}$ is defined by
%of the form (\ref{e1})
$$H_k(n)=\left|\begin{array}{ccc}
a_n & a_{n+1} \cdots & a_{n+k-1}\\
a_{n+1} & \cdots & a_{n+k}\\
\vdots & \vdots &\vdots\\
a_{n+k-1} & \cdots & a_{n+2k-2}  
\end{array}
\right|.
$$
For our discussion, in this paper, we consider the Hankel determinant $H_2(1)$ (also called the Fekete-Szeg\"o functional) and $H_2(2)$.
As early as 1916, Bieberbach himself established that if $f\in\mathcal{S}$, then $|a_2^2-a_3|\le 1$. In 1933, Fekete and Szeg\"o in \cite{FS33} proved that
$$|a_3-\mu a_2^2|\le \left \{ \begin{array}{ll}
4\mu-3 & \mbox{if } \mu \ge 1\\
1+2\exp [-2\mu/(1-\mu)] & \mbox{if } 0\le \mu \le 1\\
3-4\mu & \mbox{if } \mu \le 0
\end{array}\right. .
$$
The result is sharp in the sense that for each $\mu$ there is a function in
the class under consideration for which equality holds. 
The coefficient functional $a_3-\mu a_2^2$ has many applications in function theory. For example,
the functional $a_3-a_2^2$ is equal to $S_f(z)/6$, where $S_f(z)$ is the Schwarzian derivative 
of the locally univalent function $f$ defined by
$S_f(z)=(f''(z)/f'(z))'-(1/2)(f''(z)/f'(z))^2$.
Finding the maximum value of the functional $a_3-\mu a_2^2$ is called the
{\em Fekete-Szeg\"o problem}.
Koepf solved the Fekete-Szeg\"o problem for close-to-convex functions and obtains the largest real number $\mu$ for which $a_3-\mu a_2^2$ is maximized by the Koebe function $z/(1-z)^2$ is $\mu=1/3$ (see \cite{Koe87}). Later, in \cite{Koe87-II} (see also \cite{Lon93}), this result
was generalized for functions that are close-to-convex of order $\beta$, $\beta\ge 0$. In \cite{Pfl85}, Pfluger employed the variational method in dealing the Fekete-Szeg\"o inequality which includes a description of the image domains under extremal functions. Later, Pfluger \cite{Pfl86} used Jenkin’s method to show that
for $f\in \mathcal{S}$,
$$|a_3-\mu a_2^2|\le 1+2|\exp(-2\mu/(1-\mu))|
$$
holds for complex $\mu$ such that $\real(1/(1-\mu))\ge 1$. The inequality is sharp if and only if $\mu$
is in a certain pear shaped subregion of the disk given by
$$\mu =1-(u+itv)/u^2+v^2, \quad -1\le t\le 1,
$$ 
where $u=1-\log(\cos \varphi)$ and $v=\tan \varphi- \varphi , 0<\varphi <\pi/2$.

This section covers the Bieberbach-type problem, the Fekete-Szeg\"o problem and the Hankel determinant of functions belonging to the class $\mathcal{S}_q$. This can be extended in the similar way for the functions belonging to the class $\mathcal{S}_q(\alpha)$.
Note that the Bieberbach-type problem for the classes $\mathcal{S}_q^*$ and $\mathcal{S}_q^*(\alpha)$ 
are respectively established in the papers \cite{IMS90} and \cite{AS17}.

The following lemmas are useful for the proof of the Fekete-Szeg\"o problem and finding the Hankel determinant.

Let $\mathcal{P}$ be the family of all functions $p\in \mathcal{H}(\D)$ 
for which $\real \{p(z)\}\ge 0$ and
\begin{equation}\label{e6}
p(z)=1+p_1z+p_2z^2+\ldots
\end{equation}
for $z\in\D$.

\begin{lemma}\label{l4}\cite[pp.~254-256]{LZ83}
Let the function $p\in\mathcal{P}$ and be given by the power series (\ref{e6}). Then
$$
2p_2=p_1^2+x(4-p_1^2),
$$
$$4p_3=p_1^3+2(4-p_1^2)p_1x-p_1(4-p_1^2)x^2+2(4-p_1^2)(1-|x|^2)z,
$$
for some $x$ and $z$ satisfying $|x|\le 1$, $|z|\le 1$, and $p_1\in [0,2]$.
\end{lemma}

\begin{lemma}\label{l5}\cite[Lemma~1]{MM92}
Let the function $p\in\mathcal{P}$ and be given by the power series (\ref{e6}). Then for any real number $\lambda$,
$$
|p_2-\lambda p_1^2|\le 2 \max\{1, |2\lambda-1|\}
$$
and the result is sharp.
\end{lemma}

\begin{lemma}[Carath\'{e}dory lemma]\label{l6}
If a function $p(z)=1+\sum_{n=1}^\infty p_nz^n\in\mathcal{P}$, then $|p_n|\le 2, n=1,2, \ldots$. The result is sharp for
$$p(z)=\frac{1+z}{1-z}=1+\sum_{n=1}^\infty 2z^n.
$$ 
\end{lemma}

\subsection{The Bieberbach-type problem}
The Bieberbach-type problem for the class $\mathcal{S}_q$ is investigated and we prove the following result.
\begin{theorem}\label{T2}
If $f(z)=z+\sum_{n=2}^\infty a_nz^n\in \mathcal{S}_q$, then for all $n\ge 2$,
$$|a_n|\le \prod_{j=2}^n\left(\frac{\ds\frac{1-q^{j-1}}{1-q}+1}{\ds\frac{1-q^j}{1-q}-1}\right).
$$
Equality holds for the function $F$ satisfying $z(D_qF)(z)/F(z)=(1+z)/(1-z)$.
\end{theorem}

\begin{proof}
The proof depends on the the method of induction.
Suppose that $f\in \mathcal{S}_q$. Set
\begin{equation}\label{pf-e1}
p(z):=\frac{z(D_qf)(z)}{f(z)}=1+\sum_{n=1}^\infty p_nz^n.
\end{equation}
Clearly, $p(z)\in \mathcal{P}$. 
From \eqref{pf-e1}, we have
\begin{equation}\label{pf-e2}
z(D_qf)(z)=p(z)f(z).
\end{equation}
By substituting the series for $f(z)$ and $p(z)$ in \eqref{pf-e2}, we have
$$\sum_{n=1}^\infty\frac{1-q^n}{1-q} a_nz^n=\sum_{n=1}^\infty\left(a_n+\sum_{k=1}^{n-1}p_{n-k}a_k\right) z^n.
$$
On comparing coefficient of $z^n$ both sides, we obtain
\begin{equation}\label{e8.5}
\frac{1-q^n}{1-q} a_n=a_n+\sum_{k=1}^{n-1}p_{n-k}a_k, \quad n=1,2,\ldots,
\end{equation}
where $a_1=1$. Lemma~\ref{l6} along with the triangle inequality obtains
\begin{equation}\label{e9}
\left(\frac{1-q^n}{1-q}-1\right)|a_n|\le 2\sum_{k=1}^{n-1}|a_k|.
%\quad n=1,2,\ldots,
\end{equation}
Let us assume that the conclusion of the theorem is true for $k=2,3,\ldots, n-1$. 
That is,
\begin{equation}\label{e10}
|a_k|\le \ds \prod_{j=2}^k\left(\frac{\ds\frac{1-q^{j-1}}{1-q}+1}{\ds\frac{1-q^j}{1-q}-1}\right), \quad k=2,3,\ldots,n-1.
\end{equation}
Now our claim is to establish \eqref{e10} for $k=n$. 
Using (\ref{e10}) in (\ref{e9}) we have
\begin{eqnarray*}
\left(\frac{1-q^n}{1-q}-1\right)|a_n|&\le& 2 \left[1+\left(\frac{\ds\frac{1-q}{1-q}+1}{\ds\frac{1-q^2}{1-q}-1}\right)+\left(\frac{\ds\frac{1-q}{1-q}+1}{\ds\frac{1-q^2}{1-q}-1}\right)\left(\frac{\ds\frac{1-q^2}{1-q}+1}{\ds\frac{1-q^3}{1-q}-1}\right)+\ldots \right.\\
&&\left. \hspace*{1cm}+\left(\frac{\ds\frac{1-q}{1-q}+1}{\ds\frac{1-q^2}{1-q}-1}\right)\left(\frac{\ds\frac{1-q^2}{1-q}+1}{\ds\frac{1-q^3}{1-q}-1}\right)\ldots \left(\frac{\ds\frac{1-q^{n-2}}{1-q}+1}{\ds\frac{1-q^{n-1}}{1-q}-1}\right)\right]\\
&=& 2 \left[1+\left(\frac{2}{\ds\frac{1-q^2}{1-q}-1}\right)+\left(\frac{2}{\ds\frac{1-q^2}{1-q}-1}\right)\left(\frac{\ds\frac{1-q^2}{1-q}+1}{\ds\frac{1-q^3}{1-q}-1}\right)+\ldots \right.\\
&&\left. \hspace*{1cm}+\left(\frac{2}{\ds\frac{1-q^2}{1-q}-1}\right)\left(\frac{\ds\frac{1-q^2}{1-q}+1}{\ds\frac{1-q^3}{1-q}-1}\right)\ldots \left(\frac{\ds\frac{1-q^{n-2}}{1-q}+1}{\ds\frac{1-q^{n-1}}{1-q}-1}\right)\right]\\
&=&\frac{2\left[\left(\ds\frac{1-q^2}{1-q}+1\right)\left(\ds\frac{1-q^3}{1-q}+1\right)\ldots \left(\ds\frac{1-q^{n-1}}{1-q}+1\right)\right]}{\left(\ds\frac{1-q^2}{1-q}-1\right)\left(\ds\frac{1-q^3}{1-q}-1\right)\ldots \left(\ds\frac{1-q^{n-1}}{1-q}-1\right)}.
\end{eqnarray*}
Hence
$$|a_n|\le \ds \prod_{j=2}^n\left(\frac{\ds\frac{1-q^{j-1}}{1-q}+1}{\ds\frac{1-q^j}{1-q}-1}\right).
$$
For the proof of the equality part, 
let us assume that $F(z)=z+\sum_{n=1}^\infty b_nz^n$ such that $z(D_qF)(z)/F(z)=(1+z)/(1-z)$. Since the image of the the unit disk under the map $(1+z)/(1-z)$ is the right half plane, $F\in \mathcal{S}_q$. To find the coefficient $b_n$, write
$$z(D_qF)(z)=F(z)(1+z)/(1-z).
$$
By substituting the series representation of $F(z)$, we have
\begin{equation}\label{e11}
\sum_{n=1}^\infty \left(\frac{1-q^n}{1-q}\right)b_nz^n=\left(\sum_{n=1}^\infty b_nz^n\right)\left(1+\sum_{n=1}^\infty 2z^n\right),
\end{equation}
where $b_1=1$. Equation~(\ref{e11}) can be rewritten as
$$\sum_{n=1}^\infty \left(\frac{1-q^n}{1-q}-1\right)b_nz^n=\sum_{n=1}^\infty \left(\sum_{k=1}^{n-1}2b_k\right)z^n.
$$
Equating the coefficient of $z^n$ both sides, we have
$$\left(\frac{1-q^n}{1-q}-1\right)b_n= 2\sum_{k=1}^{n-1}b_k, \quad n=2,3,\ldots .
%\quad n=1,2,\ldots,
$$
Simple calculation leads to the conclusion that
$$b_n=\prod_{j=2}^n\left(\frac{\ds\frac{1-q^{j-1}}{1-q}+1}{\ds\frac{1-q^j}{1-q}-1}\right).
$$

This completes the proof of our theorem.
\end{proof}

\subsection{The Fekete-Szeg\"o problem}
The Fekete-Szeg\"o problem for the class $\mathcal{S}_q$ is obtained as follows:
\begin{theorem}\label{T3}
Let $f\in\mathcal{S}_q$ be of the form (\ref{e1}) and $\mu$ be any complex number. Then
$$|a_3-\mu a_2^2|\le \max\left\{\left|\frac{2(2+q)-4\mu(1+q)}{q^2(1+q)}\right|, \frac{2}{q(1+q)}\right\}.
$$
Equality occurs for the functions $F$ and $G$ satisfying
$$\frac{z(D_qF)(z)}{F(z)}=\frac{1+z}{1-z}
$$
and
\begin{equation}\label{e3}
\frac{z(D_q G)(z)}{G(z)}=\frac{1+z^2}{1-z^2}.
\end{equation}
\end{theorem}

\begin{proof}
Let $f\in \mathcal{S}_q$. By (\ref{e8.5}), we get
$$a_2=\frac{p_1}{q} \quad \mbox{ and } \quad a_3=\frac{qp_2+p_1^2}{q^2(1+q)}
$$
\begin{eqnarray*}
|a_3-\mu a_2^2|&=&\left|\frac{qp_2+p_1^2}{q^2(1+q)}-\mu \frac{p_1^2}{q^2}\right|\\
&=&\frac{1}{q(1+q)}\left|p_2-\frac{\mu(1+q)-1}{q}p_1^2\right|.
\end{eqnarray*}
We now apply Lemma~\ref{l5}, to get
$$|a_3-\mu a_2^2|\le \frac{2}{q(1+q)}\max\left\{\left|\frac{2\mu(1+q)-(2+q)}{q}\right|, 1\right\}.
$$
This completes the proof of the first part. It now remains to prove the sharpness part.
For the function $F$ defined in the statement of the theorem, it follows from the equality part of the Theorem~\ref{T2} that the n-th coefficient 
$$b_n=\prod_{j=2}^n\left(\frac{\ds\frac{1-q^{j-1}}{1-q}+1}{\ds\frac{1-q^j}{1-q}-1}\right).
$$
So, we get 
$$b_2=\frac{2}{q} \quad \mbox{ and } \quad b_3=\frac{2(2+q)}{q^2(1+q)}.
$$
Therefore,
$$|b_3-\mu b_2^2|=\left|\frac{2(2+q)-4\mu(1+q)}{q^2(1+q)}\right|.
$$
Again, it is clear that the function $G$ defined in the theorem is in the class $\mathcal{S}_q$. Also it is easy to show that the second coefficient is zero, whereas the third coefficient is $2/q(1+q)$. Hence the conclusion follows. 
\end{proof}

\begin{remark}
For $q\to 1$, Theorem~\ref{T3} gives the Fekete-Szeg\"o problem
for the class $\mathcal{S}^*$ \cite[Theorem~1]{KM69}.
\end{remark}

\subsection{The Hankel determinant}
The next result is to get an estimation for second order Hankel determinant for the class $\mathcal{S}_q$.

\begin{theorem}\label{T4}
Let $f\in\mathcal{S}_q$ be of the form (\ref{e1}). Then
$$
|H_2(2)|=|a_2a_4-a_3^2|\le \frac{4}{q^2(1+q)^2}.
$$
Equality occurs for the function $G(z)$ defined in $(\ref{e3})$.
\end{theorem}

\begin{proof}
Let $f\in \mathcal{S}_q$. By (\ref{e8.5}), we get
$$a_2=\frac{p_1}{q}, a_3=\frac{qp_2+p_1^2}{q^2(1+q)}, \mbox{ and } a_4=\frac{p_3q^2(1+q)+p_1p_2q(2+q)+p_1^3}{q^3(1+q)(1+q+q^2)}.
$$

Hence,
\begin{eqnarray*}
|a_2a_4-a_3^2|&=&\left|\frac{p_1p_3q^2(1+q)+p_1^2p_2q(2+q)+p_1^4}{q^4(1+q)(1+q+q^2)}-\frac{q^2p_2^2+2qp_1^2p_2+p_1^4}{q^4(1+q)^2}\right|.\\
&=&\left|\frac{p_1p_3}{q^2(1+q+q^2)}+\frac{(1-q)p_1^2p_2}{q^2(1+q)^2(1+q+q^2)}-\frac{p_1^4}{q^2(1+q)^2(1+q+q^2)}\right.\\
&&\left. \hspace*{9.5cm} -\frac{p_2^2}{q^2(1+q)^2}\right|
\end{eqnarray*}

Suppose now that $p_1=c$ and $0\le c\le 2$. Using Lemma~\ref{l4}, we obtain
\begin{eqnarray*}
|a_2a_4-a_3^2|&=&\left|\frac{c(c^3+2(4-c^2)cx-c(4-c^2)x^2+2(4-c^2)(1-|x|^2)z)}{4q^2(1+q+q^2)}\right.\\
&&\left. +\frac{(1-q)c^2(c^2+x(4-c^2))}{2q^2(1+q)^2(1+q+q^2)}-\frac{c^4}{q^2(1+q)^2(1+q+q^2)}-\frac{(c^2+x(4-c^2))^2}{4q^2(1+q)^2}\right|.\\
\end{eqnarray*}
Simplification yields,
\begin{eqnarray*}
|a_2a_4-a_3^2|&=&\left|\frac{(4-c^2)c^2x}{2q(1+q)^2(1+q+q^2)}+\frac{(4-c^2)(1-|x|^2)cz}{2q^2(1+q+q^2)}-\frac{(2+q)c^4}{4q^2(1+q)^2(1+q+q^2)}\right.\\
&&\left. \hspace*{6.5cm}-\frac{(4-c^2)x^2(qc^2+4(1+q+q^2))}{4q^2(1+q)^2(1+q+q^2)}\right|.
\end{eqnarray*}
Triangle inequality with $|z|\le 1$ and $\rho=|x|\le 1$ gives 
\begin{eqnarray*}
|a_2a_4-a_3^2|&\le& \left[\frac{(4-c^2)c^2\rho}{2q(1+q)^2(1+q+q^2)}+\frac{(4-c^2)(1-\rho^2)c}{2q^2(1+q+q^2)}+\frac{(2+q)c^4}{4q^2(1+q)^2(1+q+q^2)}\right.\\
&&\left. \hspace*{6.5cm}+\frac{(4-c^2)\rho^2(qc^2+4(1+q+q^2))}{4q^2(1+q)^2(1+q+q^2)}\right]\\
&\le& \left[\frac{(2+q)c^4}{4q^2(1+q)^2(1+q+q^2)}+\frac{(4-c^2)c}{2q^2(1+q+q^2)}\frac{(4-c^2)c^2\rho}{2q(1+q)^2(1+q+q^2)}\right.\\
&&\left. \hspace*{6cm}+\frac{(4-c^2)(c-2)(qc-2(1+q+q^2))\rho^2}{4q^2(1+q)^2(1+q+q^2)}\right]\\
&=& h(\rho).
\end{eqnarray*}
Furthermore, we can see that $h'(\rho)\ge 0$. This implies that $h$ is an increasing function in $\rho$ and thus the upper bound for $|a_2a_4-a_3^2|$ corresponds to the value obtained at $\rho=1$.
Hence,
$$|a_2a_4-a_3^2|\le h(1)=g(c)\, \mbox{ (say)}.
$$
Differentiation of $g$ with respect to $c$ yields,
$$g'(c)=\frac{2(1-q)c(c^2-(1-q))}{q^2(1+q)^2(1+q+q^2)}
$$
The expression $g'(c)=0$ gives either $c=0$ or $c^2=1-q$.
It can easily be verified that $g''(c)$ is negative for $c=0$ and positive for other values of $c$. Hence the maximum of $g(c)$ occurs at $c=0$. Thus, we obtain
$$
|a_2a_4-a_3^2|\le \frac{4}{q^2(1+q^2)}.
$$
%Equality holds for the function $F_1$ defined in equation $(\ref{e2})$.
The function $G$ defined in the statement of the theorem shows the sharpness of the result. 
This completes the proof of the theorem.
\end{proof}

\begin{remark}
For $q\to 1$, Theorem~\ref{T4} gives an estimation for the Hankel determinant
for the class $\mathcal{S}^*$ \cite[Theorem~3.1]{Jan07}.
\end{remark}

\vskip 1cm
\noindent
{\bf Acknowledgements.}
The author would like to thank Dr. S. K. Sahoo, Indian Institute of Technology Indore, for the useful discussions on the topic. This work has been carried out when the author was a visiting scientist at Indian Statistical Institute Chennai.

\end{document}